\documentclass[a4paper,draft,reqno,12pt]{amsart}
\usepackage[T1,T2A]{fontenc}
\usepackage[english]{babel}
\usepackage{amsmath}
\usepackage{amssymb}
\usepackage{amscd}
\usepackage{amsthm}
\usepackage{euscript}
\usepackage[all]{xy}
\usepackage{tikz,pgf}
\usetikzlibrary{matrix,arrows,decorations.pathmorphing}
\usetikzlibrary{calc}
\newtheorem{proposition}{Proposition}

\newtheorem{lemma}{Lemma}
\newtheorem{theorem}{Theorem}

\theoremstyle{definition}

\newtheorem{example}{Example}
\theoremstyle{remark}
\newtheorem {remark}{Remark}

\DeclareMathOperator{\Spec}{Spec}
\DeclareMathOperator{\End}{End}

\DeclareMathOperator{\td}{tr.deg}

\def\CC{{\mathbb C}}
\def\KK{{\mathbb K}}

\def\ZZ{{\mathbb Z}}

\def\PP{{\mathbb P}}
\def\AA{{\mathbb A}}

\def\OOO{\mathcal{O}}

\renewcommand{\phi}{\varphi}
\renewcommand{\ge}{\geqslant}
\renewcommand{\le}{\leqslant}

\begin{document}
\date{}
\title[Affine cones as images of affine spaces]{Affine cones as images of affine spaces}
\author{Ivan Arzhantsev}
\address{Faculty of Computer Science, HSE University, Pokrovsky Boulevard 11, Moscow, 109028 Russia}
\email{arjantsev@hse.ru}
\dedicatory{To Yuri Prokhorov on the occasion of his 60th birthday}
\thanks{The article was prepared within the framework of the project ``International Academic Cooperation'' HSE University} 

\maketitle
\begin{abstract}
We prove that an affine cone $X$ admits a surjective morphism from an affine space if and only if $X$ is unirational. 
\end{abstract} 

\section{Introduction}
\label{sec1}

We work over an algebraically closed field $\KK$ of characteristic zero. Recall that an irreducible algebraic variety $X$ is \emph{rational} if the field of rational functions $\KK(X)$ is isomorphic to the field of rational fractions $\KK(x_1,\ldots,x_n)$, and $X$ is \emph{unirational} if the field $\KK(X)$ is a subfield the field of rational fractions $\KK(y_1,\ldots,y_s)$ for some positive integer $s$; see~\cite[Chapter~III]{Sha}. Geometrically speaking, rationality means that there is a birational map $\AA^n\dashrightarrow X$, while unirationality is equivalent to existence of a dominant rational map $\AA^s\dashrightarrow X$.  

Assume that there is a surjective morphism $\AA^m\to X$ for some positive integer $m$. Then the variety $X$ is irreducible, unirational, and $\KK[X]^{\times}=\KK^{\times}$. One may ask whether the converse implication holds. 

In~\cite{A}, several results in this direction are obtained. Namely, it is shown that every non-degenerate toric variety, every homogeneous space of a connected linear algebraic group without non-constant invertible regular functions, and every variety covered by affine spaces admits a surjective morphism from an affine space. Further, it is proved in \cite[Theorem~1.7]{AKZ} that a complete algebraic variety $X$ is unirational if and only if there is a surjective morphism $\AA^m\to X$ for some positive integer $m$.  Moreover,  it follows from a result of Kusakabe~\cite{Ku} that the number $m$ can be taken $n+1$, where $n=\dim X$. If $\KK=\CC$, then by a result of Forstneri\v{c} there is a surjective morphism $\AA^n\to X$; see~\cite[Theorem~1.6]{Fo}. 

The proof of \cite[Theorem~1.7]{AKZ} is based on the concept of an elliptic algebraic variety in the sense of Gromov~\cite{Gro}; see \cite{Fo,AKZ,KZ,KZ-1,Ku} for more information on elliptic varieties. More precisely, let us recall that an algebraic variety $X$ is \emph{uniformly rational} if for any point $x\in X$ there is an open neighborhood $X_0$ of $x$ in $X$ isomorphic to an open subset of $\AA^n$. Clearly, any uniformly rational variety is smooth and rational. We prove in~\cite[Theorem~3.3]{AKZ} that any complete uniformly rational variety is elliptic. It follows from Chow's Lemma and Hironaka's Theorem on elimination of indeterminacy that for any complete unirational variety $X$ there is a surjective morphism $\widetilde{X}\to X$ from a uniformly rational complete variety $\widetilde{X}$. Kusakabe~\cite{Ku} proved that any elliptic variety $\widetilde{X}$ admits a surjective morphism $\AA^m\to\widetilde{X}$ with $m=\dim \widetilde{X}+1$, and \cite[Theorem~1.7]{AKZ} follows. 

Moreover, \cite[Theorem~3.3]{AKZ} claims that the punctured cone $\widehat{Y}$ over a uniformly rational subvariety $X\subseteq\PP^r$ is elliptic, where the punctured cone $\widehat{Y}$ is the affine cone $Y$ over $X$ with the vertex removed. In particular, there is a surjective morphism $\AA^{n+2}\to\widehat{Y}$, where $n=\dim X$. Note that $\widehat{Y}$ is a smooth quasi-affine, but not affine variety.

The aim of this note is to prove an analog of~ \cite[Theorem~1.7]{AKZ}  for a class of affine varieties. We say that a closed irreducible subvariety $Y\subseteq\AA^k$ is an \emph{affine cone} if $Y$ is stable under scalar multiplications. Equivalenlty, $Y$ is the affine cone over an irreducible closed subvariety $X\subseteq\PP^{k-1}$. 

\begin{theorem} \label{th1}
An affine cone $Y\subseteq\AA^k$ admits a surjective morphism $\AA^m\to Y$ for some positive integer $m$ if and only if $Y$ is unirational or, equivalently, its projectivization $X$ is unirational. Moreover, one may take $m=\dim Y +1$. 
\end{theorem}

Notice that in Theorem~\ref{th1} we can not replace an affine cone by an arbitrary affine variety, see Remark~\ref{rem} below. 
 
 \smallskip 
 
Let us come to an algebraic version of the results discussed above. Given a commutative associative algebra $A$, one may ask whether elements of $A$ can be expressed as polynomials in finitely many algebraically independent variables. Or, equivalently, can the algebra $A$ be realized as a subalgebra of the polynomial algebra $\KK[x_1,\ldots,x_m]$ for some positive integer $m$. Necessary conditions for such a realization are absence of zero divisors and absence of non-constant invertible elements. A more delicate necessary condition is that $A$ is embeddable into the field of rational fractions $\KK(y_1,\ldots,y_s)$ for some positive integer~$s$. When $A$ is finitely generated, this condition means that the affine variety $Y:=\Spec A$ is unirational. 

Let us say that a subalgebra $A\subseteq B$ is \emph{proper} if any maximal ideal in $A$ is contained in a maximal ideal of $B$. If $A$ and $B$ are finitely generated, this condition means that the corresponding morphism of affine varieties $\Spec B\to\Spec A$ is surjective. Denote by $\td A$ the transcendence degree of an algebra $A$. 

Now let $A$ be a finitely generated subalgebra without non-constant invertible elements in the field of rational fractions $\KK(y_1,\ldots,y_s)$ for some positive integer $s$. When can $A$ be properly embedded into the polynomial algebra $\KK[x_1,\ldots,x_m]$ and, if it is possible, can we take $m=\td A+1$ or even $m=\td A$? 

Clearly, an affine variety $Y$ can be realized as an affine cone in some affine space $\AA^k$ if and only if the algebra $A:=\KK[Y]$ admits a $\ZZ_{\ge 0}$-grading such that $A$ is generated by elements of degree~$1$.

\smallskip

We come to the following algebraic reformulation of Theorem~\ref{th1}. 

\begin{theorem} \label{th2}
Let $A$ be a finitely generated subalgebra in the field $\KK(y_1,\ldots,y_s)$ for some positive integer $s$. 
Assume that $A$ admits a $\ZZ_{\ge 0}$-grading such that $A$ is generated by elements of degree~$1$.
Then $A$ can be properly embedded into the polynomial algebra $\KK[x_1,\ldots,x_m]$ with $m=\td A+1$.
\end{theorem} 

It is interesting to find out whether Theorem~\ref{th2} holds for not finitely generated subalgebras, over non-closed fields, and in positive characteristic. 

\bigskip

{\bf Acknowledgments.} 
The author thanks Parnashree Ghosh, Constantin Shramov, and Mikhail Zaidenberg for useful discussions and the anonymous referee for valuable comments and references. 

\section{Proof of Theorem~\ref{th1}}
Since the ``only if'' part of Theorem~\ref{th1} is clear, we have to prove the ``if'' part. Let $Y\subseteq\AA^k$ be a unirational affine cone and $X\subseteq\PP^{k-1}$ be its projectivization. Consider the quotient morphism $\pi\colon \AA^k\setminus\{0\}\to\PP^{k-1}$ by the one-dimensional torus $T$ of scalar matrices and its restriction $\pi\colon Y\setminus\{0\}\to X$. 

By Hilbert's Theorem~90, the variety $Y$ contains an open subset $U$ isomorphic to $T\times V$, where $V$ is an open subset of $X$; see~\cite[Section~2.6]{PV}. So we have
$$
\KK(X)=\KK(V)\subseteq\KK(T\times V)=\KK(U)=\KK(Y)\subseteq\KK(y_1,\ldots,y_s), 
$$
and the variety $X$ is unirational. 

Since $X$ is complete, we deduce from~\cite[Theorem~1.7]{AKZ} that there is a surjective morphism ${\beta\colon\AA^d\to X}$, where $d=\dim X +1$. This morphism can be considered as a morphism $\beta\colon\AA^d\to\PP^{k-1}$, whose image is $X$. 

\begin{lemma} \label{lepro}
For any morphism $\beta\colon\AA^d\to\PP^{k-1}$ there is a morphism $\tilde{\beta}\colon\AA^d\to\AA^k\setminus\{0\}$ such that the following diagram is commutative:
\usetikzlibrary{matrix,arrows,decorations.pathmorphing}
     \begin{center}
        \begin{tikzpicture}[scale=2]
        \node at (1,1){$\AA^k\setminus\{0\}$};
        \node at (0,0){$\AA^d$};
        \node at (2,0){$\PP^{k-1}$.};
        \node at (1,0.1){$\beta$};
        \draw[<-][] (0.7,0.8)--node[above=-1pt]{$\tilde{\beta}$} (0.2,0.1); 
         \draw[->][] (1.3,0.8)--node[above=1pt]{$\pi$} (1.8,0.2); 
        \draw[->>][] (0.2,0)--(1.7,0);
        \end{tikzpicture}
\end{center}
\end{lemma}
\begin{proof}
Let $[z_1:\ldots:z_k]$ be homogeneous coordinates on $\PP^{k-1}$ and let $\PP^{k-1}=\bigcup_{i=1}^k U_i$ with $U_i=\{z_i\ne 0\}$ be the standard affine covering. We may assume
that $\beta(\AA^d)\cap U_1\ne\emptyset$. Then $f_i=:\beta^*(z_i/z_1)$ is a rational function on $\AA^d$. Multiplying the presentation $[1:f_2:\ldots:f_k]$ by the denominators of $f_i$
we come to the presentation $[h_1:h_2:\ldots:h_k]$. We may assume that the polynomials $h_1,h_2,\ldots,h_k$ are coprime. 

Let us check that the polynomials $h_1,h_2,\ldots,h_k$ have no common zero. Assume that 
$$
h_1(c)=h_2(c)=\ldots=h_k(c)=0
$$ 
for some $c\in\AA^d$. Let $\beta(c)\in U_i$ for some
$1\le i\le k$ and $p(x)$ be an irreducible divisor of $h_i(x)$ such that $p(c)=0$. Then there is $j\ne i$ such that $p(x)$ does not divide $h_j(x)$. Then the function
$\beta^*(z_j/z_i)=h_j(x)/h_i(x)$ is not regular at $c$, a~contradiction. This proves that the polynomials $h_1,h_2,\ldots,h_k$ define the desired morphism 
$\tilde{\beta}\colon\AA^d\to\AA^k\setminus\{0\}$.
\end{proof}

Let us return to the original morphism $\beta\colon\AA^d\to X$ and let the morphism $\tilde{\beta}\colon\AA^d\to\AA^k\setminus\{0\}$ be given by polynomials $h_1,\ldots,h_k\in\KK[\AA^d]$, which have no common zero. 
We consider the morphism
$$
\gamma\colon\AA^{d+1}=\AA^d\times\AA^1\to\AA^k, \quad (x,z)\mapsto (h_1(x)z,\ldots, h_k(x)z).
$$
Then the image $\gamma(\{a\}\times\AA^1)$ is a line in $\AA^k$ corresponding to the point $\beta(a)\in X\subseteq\PP^{k-1}$ for any $a\in\AA^d$. 
We conclude that the image of $\gamma$ is $Y$. So we obtain a surjective morphism $\gamma\colon\AA^m\to Y$ with 
$$
m=d+1=\dim X+2=\dim Y+1. 
$$
This completes the proof of Theorem~\ref{th1}. 

\begin{remark}
It follows from the proof given above that if $X\subseteq\PP^{k-1}$ is a locally closed subset that admits a surjective morphism from an affine space, then the same holds for the cone $Y$
over $X$ in $\AA^k$, where $Y$ is considered as a quasi-affine variety. 
\end{remark}

\begin{remark}
The referee observed that Lemma~\ref{lepro} can be proved in another way.
Namely, the morphism $\pi\colon\AA^k\setminus\{0\}\to\PP^{k-1}$ is a fiber bundle with fiber $\AA^1\setminus\{0\}$. 
It can be obtained from the tautological line bundle $\OOO_{\PP^{k-1}}(-1)$ over $\PP^{k-1}$ by
deleting the zero section. The line bundle $\beta^*(\OOO_{\PP^{k-1}}(-1))$ over $\AA^d$ induced
by the morphism $\beta\colon\AA^d\to\PP^k$ is trivial, since $\text{Pic}(\AA^d) = 0$. Hence, the latter
line bundle admits a non-vanishing section. This yields the desired lift $\widetilde{\beta}$ of the morphism $\beta$. 
\end{remark}

\begin{remark}
By~\cite[Theorem~3.3]{AKZ}, the punctured cone $\widehat{Y}$ over a uniformly rational subvariety $X\subseteq\PP^r$ is elliptic. This implies that there is a surjective morphism 
$\alpha\colon\AA^{n+2}\to\widehat{Y}$, where $n=\dim X$. Assume that $\alpha$ is given by polynomials $h_0,\ldots,h_r$. Consider the morphism
$$
\alpha'\colon\AA^{n+3}=\AA^{n+2}\times\AA^1\to\AA^{r+1},\quad (x,z)\mapsto (h_0(x)z,\ldots,h_r(x)z). 
$$ 
Then the image of $\alpha'$ is the affine cone $Y$ over the subvariety $X\subseteq\PP^r$. Hence~\cite[Theorem~3.3]{AKZ} implies a weaker version of Theorem~\ref{th1}. 
\end{remark}

\begin{remark}
One may try to generalize Theorem~\ref{th1} to affine varieties $Y$ that are closures in $\AA^k$ of preimages of closed unirational subvarieties $X$ in a weighted projective space
$\PP(d_1,\ldots,d_k)$ under the quotient morphism $\pi\colon\AA^k\setminus\{0\}\to\PP(d_1,\ldots,d_k)$ by the one-dimensional diagonal torus $(t^{d_1},\ldots,t^{d_k})$. But the problem is
that Lemma~\ref{lepro} does not hold in this case.

Indeed, let us consider the weighted projective plane $\PP(1,1,2)$ and its open affine chart $U_3=\{z_3\ne 0\}$ with coordinates $(z_1^2/z_3, z_1z_2/z_3,z_2^2/z_3)$. 
Note that $U_3$ is a quadratic cone given by the equation $b^2=ac$. Take the morphism
$$
\beta\colon\AA^3\to U_3\subseteq\PP(1,1,2)
$$ 
given by 
$$
(x_1,x_2,x_3)\mapsto (x_1^2x_3,x_1x_2x_3,x_2^2x_3). 
$$
The preimage $\pi^{-1}(U_3)$ is the principal open subset $W_3$ in $\AA^3\setminus\{0\}$ given by $z_3\ne 0$. Assume that there is a desired morphism $\tilde{\beta}\colon\AA^3\to\AA^3\setminus\{0\}$. This is in fact a morphism from $\AA^3$ to $W_3$.
Since the function $z_3$
is invertible on $W_3$, its image $\tilde{\beta}^*(z_3)$ is a nonzero constant $\lambda$ on~$\AA^3$. So the morphism $\tilde{\beta}$ is given as
$$
(x_1,x_2,x_3)\mapsto (h_1(x_1,x_2,x_3), h_2(x_1,x_2,x_3),\lambda), 
$$
with some polynomials $h_1$ and $h_2$. On the first coordinate of the equality $\pi(\tilde{\beta}(x))=\beta(x)$ 
we have $h_1(x_1,x_2,x_3)^2/\lambda=x_1^2x_3$, a contradiction. 
\end{remark} 

\section{Examples and applications}
\label{secc}

We begin with some examples of unirational affine cones. 

\begin{example}
Consider a hypersurface $Y$ in $\AA^{2n}$ given by the equation 
$$
x_1^{k_1}+\ldots+x_n^{k_n}+x_1y_1+\ldots+x_ny_n=0
$$
for some positive integers $k_1,\ldots,k_n$. The automorphism
$$
(x_1,\ldots,x_n,y_1-x_1^{k_1-1}, \ldots, y_n-x_n^{k_n-1})  
$$
of $\AA^{2n}$ sends $Y$ to a quadratic cone, so the variety $Y$ is rational. By Theorem~\ref{th1}, there is a surjective morphism $\AA^{2n}\to Y$. 
\end{example}

\begin{example}
Let us consider the class of so-called trinomial hypersurfaces. 
Fix a partition $k=n_0 + n_1 + n_2$ with $n_0,n_1,n_2\in\ZZ_{>0}$.
For every $i = 0,1,2$ let $l_i := (l_{i1}, \ldots , l_{in_i}) \in \ZZ_{>0}^{n_i}$
and define a monomial
$$
{\bf x}_{i}^{l_{i}}
\ := \
x_{i1}^{l_{i1}} \cdots x_{in_{i}}^{l_{in_{i}}}.
$$
The hypersurface
$$
{\bf x}_0^{l_0}+{\bf x}_1^{l_1}+{\bf x}_2^{l_2}=0
$$
in $\AA^k$ is called a \emph{trinomial hypersurface}. Let $\mathfrak{l}_i  := \gcd(l_{i1}, \ldots, l_{in_i})$. The following result characterizes rational trinomial hypersurfaces.

\begin{proposition} \cite[Proposition~5.5]{ABHW}
\label{ha}
A trinomial hypersurface $Y$ is rational if and only if one of the following conditions holds:
\begin{enumerate}
\item
there are pairwise coprime positive integers
$c_0, c_1,c_2$ and a positive integer~$s$
such that, after suitable renumbering, one has
$$
\gcd(c_2,s)=1,
\quad
\mathfrak{l}_0=sc_0,
\quad \mathfrak{l}_1=sc_1,
\quad \mathfrak{l}_2=c_2;
$$
\item
there are pairwise coprime positive integers $c_0, c_1,c_2$
such that
$$
\mathfrak{l}_0=2c_0,
\quad
\mathfrak{l}_1=2c_1,
\quad
\mathfrak{l}_2=2c_2.
$$
\end{enumerate}
\end{proposition}

Any trinomial hypersurface $Y$ carries an effective action of a torus~$T$ of dimension ${k-2}$; see \cite[Section~2]{ABHW}. By Hilbert's Theorem~90, the variety $Y$ contains an open subset isomorphic to $T\times C$, where $C$ is a curve. Since $C$ is unirational if and only if $C$ is rational, we conclude that $Y$ is unirational if and only if $Y$ is rational. 

Clearly, $Y$ is an affine cone in $\AA^k$ if and only if 
$$
\sum_{j=1}^{n_0} l_{0j}=\sum_{j=1}^{n_1} l_{1j}=\sum_{j=1}^{n_2} l_{2j}. 
$$
By Theorem~\ref{th1},  in this case Proposition~\ref{ha} gives necessary and sufficient conditions for $Y$ to admit a surjective morphism from an affine space. 

For a more general class of trinomial affine varieties (see~\cite[Construction~1.1]{ABHW} for a definition), a similar criterion of rationality is given in \cite[Corollary~5.8]{ABHW}. All the arguments used above work in this case as well. 
\end{example} 

\begin{example}
Nowadays many examples of unirational but not rational varieties are known. Among the first examples were hypersurfaces built in the work of Iskovskikh and Manin~\cite{IM}. It follows from this work that the affine cone $Y$ in $\AA^5$ given by
$$
x_1^4 + x_2^4 + x_3^4 + x_4^4 + x_1x_5^3 + x_4^3x_5 - 6x_2^2x_3^2 = 0
$$
is unirational but not rational. We conclude that there is a surjective morhism $\AA^5\to Y$. 
\end{example} 

\medskip

Let us mention some consequences of the fact that a variety admits a surjective morphism from an affine space. 

Let $X$ be an algebraic variety. We say that the monoid of endomorphisms $\End(X)$ acts on $X$ \emph{infinitely transitively}, if for any finite subset $Z\subseteq X$ and any map
$f\colon Z\to X$ there is an endomorphism $\varphi\in\End(X)$ such that $\varphi |_Z=f$. This property is a version of the infinite transitivity property for the special automorphism group, see e.g.~\cite{AFKKZ-1}.  

If $Y$ is an affine variety that admits a surjective morphism from an affine space, then the monoid $\End(Y)$ acts on $Y$ infinitely transitively; see~\cite[Corolalry~1.5]{Ku} or \cite[Proposition~5.1]{KZ-1}, where a more general result on infinite transitivity on zero-dimensional subschemes is proved. Let us give the following elementary lemma, which implies that the monoid $\End(Y)$ is infinitely transitive on $Y$. We learned this lemma from Mikhail Zaidenberg.

\begin{lemma}
Let $Z$ be a finite subset of a quasi-affine variety $X$ and $Y$ be an algebraic variety that admits a surjective morphism from an affine space. Then for any map $f\colon Z\to Y$ there is
a morphism $\varphi\colon X\to Y$ such that $\varphi |_Z=f$.
\end{lemma}

\begin{proof} Let $Z=\{z_1,\ldots,z_k\}$ and $\pi\colon\AA^m\to Y$ be a surjective morphism. Given $f\colon Z\to Y$, fix $a_1,\ldots,a_k\in\AA^m$ with $\pi(a_i)=f(z_i)$: 

\usetikzlibrary{matrix,arrows,decorations.pathmorphing}
     \begin{center}
        \begin{tikzpicture}[scale=2]
        \node at (1,1){$\AA^m$};
        \node at (0,0){$X\supseteq Z \ \ \ \ \ $};
        \node at (2,0){$Y$.};
        \node at (1,0.2){$f$};
        \draw[<-][] (0.8,0.8)--node[above=1pt]{$\tilde{\varphi}$} (0.2,0.2); 
         \draw[->>][] (1.1,0.8)--node[above=1pt]{$\pi$} (1.8,0.2); 
        \draw[->][] (0.2,0)--(1.8,0);
        \end{tikzpicture}
\end{center}

Consider $h_j\in\KK[Z]$, $j=1,\ldots,m$, where $h_j(z_i)$ is the $j$th coordinate of $a_i$. There are $f_j\in\KK[X]$ with $f_j|_Z=h_j$, and the functions $f_1,\ldots,f_m$ define a morphism $\tilde{\varphi}\colon X\to\AA^m$ with $\tilde\varphi(z_i)=a_i$. Then with $\varphi:=\pi\circ\tilde{\varphi}\colon X\to Y$ we have $\varphi |_Z=f$. 
\end{proof}

Summarizing this discussion, we come to the following result. 

\begin{proposition}
Let $Y$ be a unirational affine cone. Then the monoid $\End(Y)$ acts on $Y$ infinitely transitively. 
\end{proposition} 

\medskip

It is observed in~\cite{VBB} that by the Noether Normalization Lemma~\cite[Theorem~13.3]{Eis} any affine variety $X$ of dimension $n$ admits a surjective morphism $X\to\AA^n$. This implies that there is a surjective morphism $X\to\AA^s$ for every $s\le\dim X$. So we come to the next result, which follows from Theorem~\ref{th1} and \cite[Proposition~1.6]{VBB}; compare also with~\cite[Theorem~1.6]{Fo}. 

\begin{proposition}
Let $X$ be an affine variety and $Y$ be a unirational affine cone. If ${\dim X>\dim Y}$ then there is a surjective morphism $X\to Y$. 
\end{proposition}

\medskip

Let us make one more observation. It is clear that an affine cone $Y$ is smooth if and only if $Y$ is isomorphic to an affine space. On the other hand, an affine homogeneous space $G/H$ of a semisimple group $G$ is smooth and is not isomorphic to an affine space by~\cite[Corollary~5.1]{KP}. So $G/H$ can not be realized as an affine cone. At the same time, 
\cite[Theorem~C]{A} implies that there is a surjective morphism $\AA^m\to G/H$. 

\begin{remark} \label{rem}
There are rational affine surfaces $X$ without non-constant invertible functions that admit no surjective (and even dominant) morphism from an affine space. Namely, let $X$ be a tom~Dieck--Petrie surface or, in other words, smooth contractible affine surface of Kodaira logarithmic dimension~1. Such a surface in rational, see e.g.~\cite[Lemma~2.1]{KML}.  By~\cite[Corollary~3.3]{KML} for every two non-constant morphisms $\AA^1\to X$ their images in $X$ coincide. This implies the claim. 

More examples of rational affine varieties without non-constant invertible functions that are not dominated by affine spaces can be obtained via the technique presented in~\cite{KML-1}. 
\end{remark} 


\end{document}